\documentclass[a4paper]{article}
\usepackage{geometry}
\usepackage[tbtags]{amsmath}
\usepackage{amsthm}
\usepackage{amssymb}
\usepackage{amscd}
\usepackage{cite}
\usepackage{booktabs}
\usepackage[pdftex,colorlinks]{hyperref}
\usepackage{subfigure}
\usepackage{graphicx}
\usepackage{epstopdf}
\usepackage{epsf}
\usepackage{tikz}
\tikzset{>=stealth}
\usepackage{multicol}
\usepackage{multirow}
\usepackage{bm}
\usepackage{caption}

\numberwithin{equation}{section}

\newcommand{\norm}[1]{\left\lVert #1 \right\rVert}
\newcommand{\Hdiv}{\ensuremath{H(\text{div};\Omega)}}
\newcommand{\Hodiv}{\ensuremath{H_0(\text{div}^0;\Omega)}}
\newcommand{\Hdivo}{\ensuremath{H_0(\text{div};\Omega)}}
\newcommand{\Hcurl}{\ensuremath{H(\text{curl};\Omega)}}

\newtheorem{theorem}{\noindent{\bf Theorem}}[section]

\newtheorem{lemma}{\noindent{Lemma}}[section]

\theoremstyle{definition}
\newtheorem{remark}{\noindent{Remark}}[section]

\begin{document}
	
	\large 
	\title{{\large  \textbf{A Decoupled Low-Order Conforming Mixed Finite Element Method for a Three-Dimensional Fourth-Order Singularly Perturbed Problem}}}
	\author{\small{Yuanchun Tang}\thanks{School of Mathematics and Statistics, Yunnan University, Kunming, {\rm 650500}, China, \tt 12024113093@stu.ynu.edu.cn}   \and   \small{Baiju Zhang}\thanks{Corresponding author. School of Mathematics and Statistics, Yunnan University, Kunming, {\rm 650500}, China, \tt 20220151@ynu.edu.cn} \and \small{Zhimin Zhang}\thanks{ Department of Mathematics, Wayne State University, Detroit, MI 48202, USA, \tt ag7761@wayne.edu}}
	\date{}\maketitle 
	
	\begin{abstract}
		This paper develops a decoupled low-order conforming finite element method for a fourth-order elliptic singular perturbation problem in three dimensions. By means of a generalized Helmholtz decomposition, the problem is reduced to two second-order elliptic problems and a system of generalized singularly perturbed Stokes-type equations subject to a curl-free constraint. The former are discretized by standard linear finite elements. For the latter, we employ the MINI element and show that, after adding an  $L^2$ term involving a Lagrange multiplier, the resulting discretization becomes robust with respect to the perturbation parameter. We further establish an error estimate of order $h^{1/2}$ uniform with respect to the perturbation parameter. Numerical experiments are included to support the theory.
	\end{abstract}
	
	
	\section{Introduction}
	Let $\Omega\subset \mathbb{R}^3$ denote a bounded domain, and its boundary be denoted as $\partial \Omega$. This paper studies a low-order decoupled conforming mixed finite element method for a fourth-order singular perturbation problem:
\begin{equation}\label{equation1}
\begin{cases}
	\varepsilon^2 \Delta^2 u-\Delta u=f & \text{in } \Omega, \\
	u=\partial_n u=0 & \text{on } \partial \Omega.
\end{cases} 
\end{equation}
 	 Here, $f\in L^2(\Omega)$ is a given function, $\Delta^2$ represents the biharmonic operator, $\partial_n$ denotes the normal derivative at the boundary and $\varepsilon >0$ is a real parameter. 

Problem \eqref{equation1} was considered in the early work \cite{semper1992conforming} using $H^2$-conforming finite elements; see \cite{Chen2025ImplementationAB, chen2021geometricdecompositionssimpliciallattice, MR4835148, MR2474112}  for further developments of such elements. However, due to the complexity in implementing such conforming elements, various $H^2$-nonconforming elements are more commonly used in practice \cite{MR3039784, MR3267353, MR2204450, MR2431184, MR2359954}. In addition, the $C^0$ interior penalty discontinuous Galerkin method \cite{MR2431184}, which is based on Lagrange element spaces, has also been adopted in some studies. Most of the aforementioned methods are designed for the primal formulation of problem \eqref{equation1} \cite{MR3039784, MR3267353, MR2204450, MR2359954}. 

In recent years, there has been growing interest in finite element methods based on the Helmholtz decomposition for fourth-order singularly perturbed elliptic problems. Huang et al. used the Morley-Wang-Xu element with a modified right-hand side to overcome the divergence phenomenon as $\varepsilon \to 0$, and introduced Nitsche's method to achieve optimal second-order convergence in the energy norm in the presence of boundary layers \cite{MR4252891}. Chen et al. developed a decoupling framework via generalized Helmholtz decomposition, reducing the problem to two Poisson equations and a Brinkman-type equation \cite{MR3852718}. More recently, Cui et al. constructed a low-order non-conforming smooth de Rham complex with a N\'ed\'elec interpolation operator, attaining optimal uniform convergence without Nitsche's method  \cite{cui2025loworderfiniteelementcomplex}. 

A common feature of these methods is that they first apply the Helmholtz decomposition to \eqref{equation1} to reformulate the original problem into a decomposed system, and then design finite element schemes for the resulting formulation. Specifically, \eqref{equation1} can be rewritten as follows \cite{cui2025loworderfiniteelementcomplex}:
\begin{subequations}\label{equation2}
\begin{equation}\label{equation2.1}
	(\nabla w, \nabla v) = (f, v) \quad \forall v \in H_0^1(\Omega), 
	\end{equation}
	\begin{equation}\label{equation2.2}
	\varepsilon^2(\nabla \boldsymbol{\phi}, \nabla \boldsymbol{\psi}) + (\boldsymbol{\phi}, \boldsymbol{\psi}) + (\operatorname{curl} \boldsymbol{\psi}, \boldsymbol{p}) = (\nabla w, \boldsymbol{\psi}) \quad \forall \boldsymbol{\psi} \in H_0^1(\Omega; \mathbb{R}^3),  
	\end{equation}
	\begin{equation}\label{equation2.3}
	(\operatorname{curl} \boldsymbol{\phi}, \boldsymbol{q}) = 0 \quad \forall \boldsymbol{q} \in L^2(\Omega; \mathbb{R}^3) / \nabla H^1(\Omega),  
	\end{equation}
	\begin{equation}\label{equation2.4}
	(\nabla u, \nabla \xi) = (\boldsymbol{\phi}, \nabla \xi) \quad \forall \xi \in H_0^1(\Omega). 
\end{equation}
\end{subequations}
The above system avoids the need to construct $H^2$-conforming finite element spaces and can be solved sequentially. This technique of decoupling high-order partial differential equations into several low-order ones based on the Helmholtz decomposition was successfully applied to the Reissner-Mindlin plate model as early as the 1980s \cite{MR842127}. Nowadays, this technique has been widely extended to the research and computation of various high-order partial differential equations, such as fourth-order elliptic equations, fourth-order curl equations, sixth-order elliptic equations, and transmission eigenvalue problems \cite{MR3533246, MR3667017, MR3852718, MR3745016, MR3808156, MR4345353, MR4638728, MR4867672, cui2025loworderfiniteelementcomplex, MR4890863, gu2025nonconforminglinearelementmethod}. Note that \eqref{equation2.2} and \eqref{equation2.3} are analogous to the corresponding equations in the Brinkman problem. The presence of small parameters in \eqref{equation2.2} and \eqref{equation2.3} suggests that ideas used to design robust numerical schemes for the Brinkman problem may also be applicable here. In fluid dynamics, an important approach to constructing robust numerical schemes is to design finite element spaces that preserve the divergence-free constraint \cite{MR2283095,xie2008uniformly,guzman2012afamily}. 
By analogy, the corresponding issue in the present setting is to preserve the curl-free constraint.
This is precisely the approach adopted in \cite{cui2025loworderfiniteelementcomplex}. 

Despite the appeal of such an approach, the exact preservation of the divergence-free or curl-free constraint at the discrete level generally requires carefully designed finite element spaces. This motivates us to seek an alternative approach that is simpler to implement while still retaining the desired stability and robustness. This consideration brings to mind the use of the MINI element in \cite{MR2672618}. In two dimensions, extending this idea to Problem \eqref{equation1} is rather straightforward, since the divergence and curl operators differ only by a $\pi/2$ rotation. In three dimensions, however, this relation no longer holds. This raises the question of whether the same idea can be extended to the Brinkman-type equations subject to a curl-free constraint considered here. To illustrate how the aforementioned difficulty can be overcome, we first rewrite \eqref{equation2.2} and \eqref{equation2.3} as follows:
\begin{equation}\label{equation3} 
	\left\{ \begin{aligned} -\varepsilon^2 \Delta \boldsymbol{\phi} + \boldsymbol{\phi} +\operatorname{curl} \boldsymbol{p} &= \nabla w ,\\ \nabla \cdot \boldsymbol{p} &= 0 ,\\ \operatorname{curl} \boldsymbol{\phi}- \nabla r &= \boldsymbol{0}. \end{aligned} \right. 
	\end{equation}
	where \(0 \leq \varepsilon \leq C\).  For \(\varepsilon > 0\), assumed that boundary conditions are satisfied $\boldsymbol{\phi} |_{\partial\Omega} = 0 $ \hspace{3pt}and\hspace{3pt}	$\boldsymbol{p} \cdot \boldsymbol{n}|_{\partial\Omega} = 0.$
	For\hspace{3pt}$\varepsilon=0$\hspace{3pt}, it has another boundary conditions\hspace{3pt}$\boldsymbol{\phi} \times \boldsymbol{n} |_{\partial\Omega} = 0 \hspace{3pt}\text{and}\hspace{3pt} \boldsymbol{p}\cdot \boldsymbol{n}|_{\partial\Omega} = 0.$
As shown in \cite{cui2025loworderfiniteelementcomplex}, the Lagrange multiplier $r$ in the above system is identically zero. At this point, the generalized Stokes equation system can be written as:
 \begin{equation}\label{equation4}
 \left\{ \begin{array}{c c c c} -\varepsilon^2 \Delta \boldsymbol{\phi} + \boldsymbol{\phi} & & +\operatorname{curl} \, \boldsymbol{p} & = \nabla w ,\\ & r & +\nabla \cdot \boldsymbol{p} & = 0 ,\\ \operatorname{curl} \boldsymbol{\phi} & - \nabla r & & = \boldsymbol{0}. \end{array} \right.
 \end{equation}
The resulting formulation admits a MINI element discretization for which we establish the corresponding discrete inf-sup condition and derive error estimates that are robust with respect to the perturbation parameter $\varepsilon$. Moreover, the resulting method admits a simpler implementation.

The rest of this paper is organized as follows. In Section \ref{sec:2}, we introduce the notation, analyze the decoupled generalized Stokes-type system, and prove that, upon using the relation $r=0$, its variational formulation is equivalent to the original one. In Section \ref{sec:3}, we presents the MINI element method for the three-dimensional fourth-order singular perturbation problem, including the finite element space and interpolation, as well as error estimation. Finally, in Section \ref{sec:4}, we report numerical results that confirm the theoretical analysis.

\section{A modified decoupled formulation for the fourth-order singularly perturbation problem }\label{sec:2}
In this section, we apply the analysis method for the Brinkman problem \cite{MR2672618} to the generalized Stokes problem and thereby propose a decoupled finite element method for the fourth-order singularly perturbed problem.

\subsection{Notation}
	Let $\Omega \subset \mathbb{R}^3$ be a bounded polyhedron. Given an integer $m \geq 0$ and a bounded region $K \subset \mathbb{R}^3$, we define $H^m(K)$ as the standard Sobolev space of functions on $K$. The corresponding norm and semi-norm are denoted as $\|\cdot\|_{m,K}$ and  $ |\cdot|_{m,K}$ respectively. Let $L^2(M) = H^0(K)$. For an integer $k \geq 0$, let $P_k(K)$ represent the space consisting of all polynomials on $M$ with total degree not exceeding $k$. For $k<0$, define $P_k(K) = {0}$. Let $L^2_0(K)$ be the space of functions in $L^2(K)$ with zero integral average. For a space $B(K)$ defined on $K$, denote its vector version as $B(K;\mathbb{R}^3) := B(K) \otimes \mathbb{R}^3$. We use $(\cdot,\cdot)_K$ to represent the usual inner product in $L^2(K)$ or $L^2(K;\mathbb{R}^3)$. Let $H_0^m(K)$ (or $H_0^m(K;\mathbb{R}^3)$) be the closure of $C^\infty_0(K)$ (or $C^\infty_0(K;\mathbb{R}^3)$) under the norm $\|\cdot\|_{m,K}$. When $K$ is exactly $\Omega$, we will simply denote $\|\cdot\|_{m,K}$, $|\cdot|_{m,K}$ and $(\cdot,\cdot)_K $as $\|\cdot\|_m$, $|\cdot|_m $and $(\cdot,\cdot)_m$, respectively.
	
	We denote the gradient operator, curl operator and divergence operator as $\nabla$ ,curl and div, respectively. Introduce Sobolev spaces
    \begin{align*} H(\text{div}; K) &:= \{ \boldsymbol{v} \in L^2(K; \mathbb{R}^3) : \operatorname{div} \boldsymbol{v} \in L^2(K;\mathbb{R}^3) \}, \\
	 H_0(\text{div}^0; K) &:= \{ \boldsymbol{v} \in H(\text{div}; K) : \boldsymbol{v} \cdot \boldsymbol{n} = 0 \text{ on } \partial K  \text{and}  \operatorname{div} \boldsymbol{v}=0 \}, \\
	 H(\text{curl}; K) &:= \{ \boldsymbol{v} \in L^2(K; \mathbb{R}^3) : \operatorname{curl} \boldsymbol{v} \in L^2(K; \mathbb{R}^3) \}, \\
	 H_0(\text{curl}; K) &:= \{ \boldsymbol{v} \in H(\text{curl}; K) : \boldsymbol{v} \times \boldsymbol{n} = 0 \text{ on } \partial K \}. 
	\end{align*}
  Let $H(\operatorname{div},\operatorname{curl};\Omega)=H_0(\operatorname{div};\Omega)\cap H(\operatorname{curl};\Omega)$, by \cite[Lemma 3.6]{Vivette1986finite} we can equip this space with the following norm \[\|\boldsymbol{q}\|_{H(\operatorname{div},\operatorname{curl};\Omega)}^2=\|\operatorname{div}\boldsymbol{q}\|_0^2+\|\operatorname{curl}\boldsymbol{q}\|_0^2.\]
	
    In this paper, we use  ``$\lesssim \cdots$'' to mean that `` $\leq C \cdots $'', where  $C$  is a generic positive constant independent of  $h$ , which may take diﬀerent values in diﬀerent contexts. Moreover, $A\eqsim B$ means that $A \geq B$ and $A \leq B$.
	
	\subsection{The modified decoupled formulation of the generalized Stokes problem}
	Next, we consider the variational form of \eqref{equation3}. Inspired by \cite{MR2672618} , we consider the following spaces and norms.
	
   The natural energy norm of $\boldsymbol{\phi}$ is 
	\begin{equation}\label{Stokes1}
	\|\boldsymbol{\phi}\|_\varepsilon^2=\varepsilon^2 \|\nabla \boldsymbol{\phi}\|_0^2 + \|\boldsymbol{\phi}\|_0^2,
	\end{equation}
	and the natural solution space is $V$, the completion of $C_0^\infty(\Omega;\mathbb R^3)$\hspace{3pt}with respect to this norm.
	
	For $\varepsilon>0$,
	\begin{equation}\label{Stoeks2}
	V = H_0^1(\Omega;\mathbb R^3),\
	\end{equation}
	but the equivalence is not uniform, for\hspace{3pt}$0 < \varepsilon \leq C$ \hspace{3pt}it holds
	\begin{equation}\label{Stokes3}
	C_1 \varepsilon \|\boldsymbol{\psi}\|_1 \leq \|\boldsymbol{\psi}\|_\varepsilon \leq C_2 \|\boldsymbol{\psi}\|_1, \quad \forall \boldsymbol{\psi} \in V.
	\end{equation}
	For $\varepsilon=0$ the space is
	\begin{equation}\label{Stokes4}
	V = L^2(\Omega;\mathbb R^3).
	\end{equation}
	The space for the Lagrange multiplier $r$ is \begin{equation}
	 R = L_0^2(\Omega).
	\end{equation}
	For $(\boldsymbol{\psi},s) \in V \times R$, we can define
	\begin{equation}\label{Stokes6}
	\|(\boldsymbol{\psi},s)\|_\varepsilon^2 = \|\boldsymbol{\psi}\|_\varepsilon^2 + \|s\|_0^2.
	\end{equation}
	The space for $\boldsymbol{q}$ is defined through the norm, denoted by $Q$:
	\begin{equation}\label{Stokes7}
	\|\boldsymbol{q}\|_\varepsilon^2 = \left( \sup_{\boldsymbol{\psi} \in V} \frac{\langle \boldsymbol{\psi}, \operatorname{curl} \boldsymbol{q} \rangle}{\|\boldsymbol{\psi}\|_\varepsilon} \right)^2 + \|\operatorname{div} \boldsymbol{q}\|_0^2,
	\end{equation}
	where $\langle \cdot, \cdot \rangle$ denotes the duality pairing in $V \times V^*$.
	
	We set
	\begin{equation}\label{Stokes8}
	Q = \{ \boldsymbol{q} \in H_0(\operatorname{div};\Omega) \mid \|\boldsymbol{q}\|_\varepsilon < \infty \}.
	\end{equation}
	Note that for $(\boldsymbol{\psi},\boldsymbol{q}) \in V \times Q$ it holds
	\begin{equation}\label{Stokes9}
	\langle \boldsymbol{\psi}, \operatorname{curl} \boldsymbol{q} \rangle =
	\begin{cases}
		(\operatorname{curl} \boldsymbol{\psi}, \boldsymbol{q}), & \varepsilon>0, \\
		(\boldsymbol{\psi}, \operatorname{curl} \boldsymbol{q}), & \varepsilon=0.
	\end{cases} 
	\end{equation}
	For $\varepsilon>0$, we need a Lemma to verify the Babu\v{s}ka-Brezzi condition.
	
	\begin{lemma}\label{lemmain}
		 There is a constant\hspace{3pt}$\beta>0$\hspace{3pt}, such that the following inf-sup condition holds for\hspace{3pt}$\varepsilon>0$\hspace{3pt}:
	\begin{equation}\label{inf-sup1}
	\sup_{\boldsymbol{0} \neq(\boldsymbol{\psi},s)\in V\times R} 
	\frac{(\operatorname{curl}\boldsymbol{\psi},\boldsymbol{q}) + (s,\operatorname{div}\boldsymbol{q})}{\|({\boldsymbol{\psi}},s)\|_\varepsilon}
	\geq \beta \norm{{\boldsymbol{q}}}_{H(\operatorname{div};\Omega)},
	\quad \forall {\boldsymbol{q}}\in{Q}.
	\end{equation}
\end{lemma}
	\begin{proof}
		 According to \cite[Lemma\ 3.4]{Vivette1986finite},$\quad\forall \boldsymbol{q}\in Q\subset L^2(\Omega;\mathbb{R}^3)$\hspace{3pt}, $\exists $\hspace{3pt}$\rho\in (H^1(\Omega)/\mathbb{R})$\hspace{3pt}and\hspace{3pt}$ \hat{\boldsymbol{q}}\in\Hodiv$\hspace{3pt}such that
	\[
	\boldsymbol{q} = \nabla \rho + \hat{\boldsymbol{q}}.
	\]
	Since $\hat{{\boldsymbol{q}}}\in\Hodiv$, by \cite[proposition\ 13]{MR3667017}, for any $\hat{\boldsymbol{q}}\in Q$ there exists $\hat{\boldsymbol{\psi}}\in V$ with $ \operatorname{curl} \hat{\boldsymbol{\psi}}=\hat{\boldsymbol{q}}$  and  $\|\hat{\boldsymbol{\psi}}\|_1\leq C\|\hat{\boldsymbol{q}}\|_{H(\operatorname{div};\Omega)}$.  Due to  $\boldsymbol{q} \in \Hdivo$, we have
	\[
	\nabla \rho \in \Hdivo \cap \Hcurl.
	\]
	Then by \cite[Lemma\ 3.4]{Vivette1986finite}, we obtain
	\[
	\norm{\nabla \rho} \leq C \norm{\operatorname{div} \boldsymbol{q}}.	\]
	Setting $s = 2\operatorname{div} \boldsymbol{q}$, we can get
\begin{align*}
	(\operatorname{curl} \hat {\boldsymbol{\psi}}, \boldsymbol{q}) + (s, \operatorname{div} \boldsymbol{q}) & \geq \|\hat{\boldsymbol{q}}\|_0^2 + 2\|\operatorname{div} \boldsymbol{q}\|_0^2, \\ 
	& \geq \|\hat{\boldsymbol{q}}\|_0^2 + C^{-1}\|\nabla \rho\|_0^2 + \|\operatorname{div} \boldsymbol{q}\|_0^2, \\
	& \geq \beta \|\boldsymbol{q}\|_{\Hdiv}^2.
\end{align*}
	Therefore, the desired result follows from the above inequality, $\|\hat{\boldsymbol{\psi}}\|_1 \leq C_1\|\hat{\boldsymbol{q}}\|_{\Hdiv} \leq C\|\boldsymbol{q}\|_{\Hdiv}$ and $\|s\|_0 \leq 2\|\boldsymbol{q}\|_{\Hdiv}$.  
	\end{proof}  
	
	Lemma \ref{lemmain} implies that $Q = H_0(\operatorname{div};\Omega)$ for $\varepsilon>0$.
	
	\begin{lemma} \label{lemmaimp}
		There is a constant $\beta>0$, such that the following inf-sup condition holds for $\varepsilon=0$:
	\begin{equation}\label{stokes2}
	\sup_{\boldsymbol{0} \neq(\boldsymbol{\psi},s)\in{V}\times\mathbb{R}}
	\frac{(\boldsymbol{\psi},\operatorname{curl} \boldsymbol{q}) + (s,\operatorname{div} \boldsymbol{q})}{\norm{( \boldsymbol{\psi},s)}_\varepsilon}
	\geq \beta \norm{\boldsymbol{q}}_{H(\operatorname{curl},\operatorname{div};\Omega)},
	\quad \forall \boldsymbol{q}\in H(\operatorname{curl};\Omega)\cap H_0(\operatorname{div};\Omega).
	\end{equation}
\end{lemma}
	\begin{proof} 
		Clearly, for all $\boldsymbol{q}\in \Hdivo \cap \Hcurl$, we have
	\[
	\operatorname{curl} \boldsymbol{q}  \in L^2(\Omega;\mathbb R^3) \hspace{3pt}\text{and}\hspace{3pt}\quad \operatorname{div} \boldsymbol{q}  \in L^2(\Omega).
	\]
	Taking $ \boldsymbol{\psi}= \operatorname{curl} \boldsymbol{q}$ and $s = \operatorname{div} \boldsymbol{q}$ gives
	\[
	(\boldsymbol{\psi} ,\operatorname{curl} \boldsymbol{q}) + (s,\operatorname{div} \boldsymbol{q}) = \|\operatorname{curl} \boldsymbol{q}\|_0^2 + \|\operatorname{div} \boldsymbol{q}\|_0^2.
	\]
	Then the inequality follows from \cite[Lemma 3.6]{Vivette1986finite}.  \end{proof}
	
	Meanwhile, Lemma \ref{lemmaimp} implies that for $\varepsilon=0$ we have
	\[
	\|\boldsymbol{q}\|_\varepsilon = \norm{\boldsymbol{q}}_{H(\text{curl},\text{div};\Omega)} 
	\]
	and $Q= \Hdivo \cap \Hcurl$.
	
	Define the bilinear forms
	\begin{subequations}\label{stokes3}
	\begin{align}
		a_1(\boldsymbol{\phi};\boldsymbol{\psi}) &= \varepsilon^2 (\nabla\boldsymbol{\phi},\nabla\boldsymbol{\psi}) + (\boldsymbol{\phi},\boldsymbol{\psi}), \\
		a_2(r,s) &= (r,s), \\
		b_1(\boldsymbol{\phi},q) &= \langle \boldsymbol{\phi}, \nabla\times \boldsymbol{q} \rangle, \\
		b_2(r,\boldsymbol{q}) &= (r,\nabla\cdot \boldsymbol{q}),
	\end{align}
	\end{subequations}
	and
\[
B_1(\boldsymbol{\phi}, r, \boldsymbol{p}; \boldsymbol{\psi}, s, \boldsymbol{q})= a_1(\boldsymbol{\phi}, \boldsymbol{\psi}) + b_1(\boldsymbol{\phi}, \boldsymbol{q}) + b_2(\boldsymbol{\psi}, \boldsymbol{p}) + b_1(r, \boldsymbol{q}) + b_2(s, \boldsymbol{p}).
	\]
	Then the variational problem of \eqref{equation3} reads: find $u,w \in H_0^1(\Omega)$, $\boldsymbol{\phi} \in V$, $\boldsymbol{p}\in Q$, $r\in R$ such that
	\begin{subequations}\label{stokes4}
	 \begin{equation}\label{stokes4.1}
	(\nabla w,\nabla v)=(f,v), 
	\end{equation}
	\begin{equation}\label{stokes4.2}
B_1(\boldsymbol{\phi}, r, \boldsymbol{p}; \boldsymbol{\psi}, s, \boldsymbol{q})=\mathcal{L}(\boldsymbol{\psi},s,\boldsymbol{q}), 
\end{equation}
\begin{equation}\label{stokes4.3}
	(\nabla u,\nabla \xi)=(\boldsymbol{\phi},\nabla \xi),
	\end{equation}
\end{subequations}
where $\mathcal{L}(\boldsymbol{\psi},s,\boldsymbol{q})=(\nabla w, \boldsymbol{\psi})$, for all $v, \xi \in H_0^1(\Omega)$, $\boldsymbol{\psi} \in V$, $\boldsymbol{q} \in Q$, $s \in R$.	Based on the spaces and norms defined earlier, we can establish the corresponding inf-sup condition and well-posedness. To this end, we need to utilize the following Lemma \ref{lemmazero}.
	
	\begin{lemma}\label{lemmazero}
		 For $\varepsilon \ge 0$, if $(\boldsymbol{\psi}, s) \in W= \left\{ (\boldsymbol{\psi}, s) \in V \times R \mid b_1(\boldsymbol{\psi}, \boldsymbol{q}) + b_2(s, \boldsymbol{q}) = 0,\ \forall \boldsymbol{q} \in Q \right\}$, then $s = 0$.\end{lemma}
	
	\begin{proof} 
		For $s \in R$, let
	\[
	-\Delta \rho = -s,\quad \frac{\partial \rho}{\partial n}|_{\partial \Omega} = 0,
	\]
	which implies $\nabla \rho \in Q$.   Through the definition of $W$, we get
	\[
	(\boldsymbol{\psi}, \operatorname{curl} \nabla \rho) + (s, \operatorname{div} \nabla \rho) = 0.
	\]
	By the definition of $\rho$, we obtain $s = 0$.  \end{proof}
Especially, by the above lemma we can derive $r \equiv 0$  when $\varepsilon \geq 0$.
	
Through Lemma \ref{lemmain}-Lemma \ref{lemmazero} and the definition of 
$\|\boldsymbol{q}\|_\varepsilon$, the following inequalities hold for $\varepsilon \geq 0$:
	\begin{equation}
		\begin{aligned}
	&a_1(\boldsymbol{\phi}, r; \boldsymbol{\phi}, r) \ge C\left(\norm{\boldsymbol{\phi}}_\varepsilon^2 + \norm{r}_0^2\right),  \quad \forall (\boldsymbol{\phi}, r) \in W,  \\
	 &\sup_{\substack{\boldsymbol{0} \neq (\boldsymbol{\psi},s) \in V \times R}} \frac{b_1(\boldsymbol{\psi}, \boldsymbol{q}) + b_2(s, \boldsymbol{q})}{\|(\boldsymbol{\psi},s)\|_\varepsilon} \geq \beta \|\boldsymbol{q}\|_\varepsilon.
	\end{aligned}
	\end{equation}
	These two imply the stability condition 
	\begin{equation}\label{stokes5}
	\sup_{\boldsymbol{0} \neq (\boldsymbol{\psi},s,\boldsymbol{q}) \in V \times R \times Q} \frac{B_1(\hat{\boldsymbol{\phi}}, \hat{r}, \hat{\boldsymbol{p}}; \boldsymbol{\psi},s,\boldsymbol{q})}{\|(\boldsymbol{\psi},s)\|_\varepsilon + \|\boldsymbol{q}\|_\varepsilon} \gtrsim \|(\hat{\boldsymbol{\phi}} , \hat{r})\|_\varepsilon + \|\hat{\boldsymbol{p}}\|_\varepsilon,  \quad \forall  (\hat{\boldsymbol{\phi}}, \hat{r}, \hat{\boldsymbol{p}}) \in V \times R \times Q.
	\end{equation}
By the definitions of \eqref{Stokes7} and \eqref{Stokes9}, we can easily get the boundedness of $B_1(\cdot;\cdot)$ as follows.
	\begin{theorem}
	There exists a constant  $C>0$  such that
		\begin{equation}\label{boundnessB1}
	B_1(\tilde{\boldsymbol{\phi}},\tilde{r},\tilde{\boldsymbol{p}};\boldsymbol{\psi},s,\boldsymbol{q}) \leq C(\|\tilde{\boldsymbol{\phi}}\|_\varepsilon+\|\tilde{r}\|_0+\|\tilde{\boldsymbol{p}}\|_\varepsilon)\left(\|\boldsymbol{\psi}\|_\varepsilon+\|s\|_0+\|\boldsymbol{q}\|_\varepsilon\right),
			\end{equation}
    for all $(\tilde{\boldsymbol{\phi}},\tilde{r},\tilde{\boldsymbol{p}}),(\boldsymbol{\psi},s,\boldsymbol{q})\in V \times R \times Q.$
	\end{theorem}
Clearly, \eqref{stokes5} and \eqref{boundnessB1} imply that the solution of \eqref{stokes4} is unique. 
    
	We are now in a position to state the modified version of \eqref{stokes4} : find $u,w \in H_0^1(\Omega)$, $\boldsymbol{\phi} \in V$, $\boldsymbol{p}\in Q$, $r\in R$ such that
	\begin{subequations}\label{stokes6}
	\begin{equation}\label{stokes6.1}
		(\nabla w,\nabla v)=(f,v), 
\end{equation}
\begin{equation}\label{stokes6.2}
	B(\boldsymbol{\phi}, r, \boldsymbol{p}; \boldsymbol{\psi}, s, \boldsymbol{q})=\mathcal{L}(\boldsymbol{\psi},s,\boldsymbol{q}),
	\end{equation}
	\begin{equation}\label{stokes6.3}
		(\nabla u,\nabla \xi)=(\boldsymbol{\phi},\nabla \xi),
	\end{equation}
	\end{subequations}
	where $B(\boldsymbol{\phi}, r, \boldsymbol{p};\boldsymbol{\psi} ,s,\boldsymbol{q}) = B_1(\boldsymbol{\phi}, r, \boldsymbol{q}; \boldsymbol{\psi},s,\boldsymbol{q}) + a_2(r,s)$, for all $v, \xi \in H_0^1(\Omega)$, $\boldsymbol{\psi} \in V$, $\boldsymbol{q} \in Q$, $s \in R$.
	
	Similarly, we also have the following inf-sup condition:
	\begin{equation}\label{stokes7}
	\sup_{\substack{\boldsymbol{0} \neq (\boldsymbol{\psi},s,\boldsymbol{q}) \in V \times R \times Q}} \frac{B(\hat{\boldsymbol{\phi}}, \hat{r}, \hat{\boldsymbol{p}}; \boldsymbol{\psi},s,\boldsymbol{q})}{\|(\boldsymbol{\psi},s)\|_\varepsilon + \|\boldsymbol{q}\|_\varepsilon} \gtrsim \|(\hat{\boldsymbol{\phi}}, \hat{r})\|_\varepsilon + \|\hat{\boldsymbol{p}}\|_\varepsilon,  \quad \forall  (\hat{\boldsymbol{\phi}}, \hat{r}, \hat{\boldsymbol{p}}) \in V \times R \times Q,
	\end{equation}
    and the boundedness of $B(\cdot;\cdot)$:
    \begin{equation}\label{boundnessB}
	B(\tilde{\boldsymbol{\phi}},\tilde{r},\tilde{\boldsymbol{p}};\boldsymbol{\psi},s,\boldsymbol{q}) \leq C(\|\tilde{\boldsymbol{\phi}}\|_\varepsilon+\|\tilde{r}\|_0+\|\tilde{\boldsymbol{p}}\|_\varepsilon)\left(\|\boldsymbol{\psi}\|_\varepsilon+\|s\|_0+\|\boldsymbol{q}\|_\varepsilon\right),
			\end{equation}
    for all $(\tilde{\boldsymbol{\phi}},\tilde{r},\tilde{\boldsymbol{p}}),(\boldsymbol{\psi},s,\boldsymbol{q})\in V \times R \times Q.$    
    
	It is clear from $r\equiv 0$ that the solution of \eqref{stokes4} is also the solution of \eqref{stokes6}. Then \eqref{stokes7} and \eqref{boundnessB} imply that \eqref{stokes4} and \eqref{stokes6} are equivalent. The subsequent part of the article will be based on \eqref{stokes6} for spatial discretization, by which we can establish robust mixed method based on MINI element.
	
	\section{Mixed finite element methods}\label{sec:3}
	Let \( \mathcal{T}_h \) be a triangulation of \( \Omega \). For each \( K \in \mathcal{T}_h \), let \( h_K = \operatorname{diam}(K) \) and define the mesh size $h = \max_{K \in \mathcal{T}_h} h_K$.
	We assume that the family of triangulations $\{\mathcal{T}_h\}_{h>0}$ is shape-regular and quasi-uniform. That is, there exists a constant $\rho > 0$ such that for all $K \in \mathcal{T}_h$,
	$h_K \le \rho \, \sigma_K$, where $\sigma_K$ is the diameter of the largest ball inscribed in $K$, and there exists a constant $c > 0$ such that $h \le c \min_{K \in \mathcal{T}_h} h_K$.
	
Then, we consider the following  several finite element spaces
	\begin{subequations}\label{finite1}
	\begin{align}
		U_h &=\left\{v_h \in H_0^1(\Omega) \mid v_h \mid \in P_1(K)  \quad \forall K \in \mathcal{T}_h\right\},\\
		V_h &= \left\{ \boldsymbol{\psi}_h \in V \mid \boldsymbol{\psi}_h|_K \in [P_1(K) \oplus B_4(K)]^3  \quad \forall K \in \mathcal{T}_h \right\}, \\
		R_h &= \left\{ s_h \in R \mid s_h|_K \in P_0(K) \quad  \forall K \in \mathcal{T}_h \right\}, \\
		Q_h &= \left\{ \boldsymbol{q}_h \in \Hdivo \cap H^1(\Omega;\mathbb{R}^3) \mid \boldsymbol{q}_h|_K \in P_1(K) \quad  \forall K \in \mathcal{T}_h \right\},
	\end{align}
	\end{subequations}
where
	\[
	B_4(K) = P_4(K) \cap H_0^1(K)
	\]
	consists of the polynomial bubble functions of degree at most 4. We also introduce the subspace $\overline{V}_h \subset V_h$, obtained by omitting the $B_4(K)$ bubble functions:
	\begin{equation}
	\overline{V}_h = \left\{ \boldsymbol{\psi}_h \in V \mid \boldsymbol{\psi}_h|_K \in P_1(K;\mathbb{R}^3)  \quad \forall K \in \mathcal{T}_h \right\}.
	\end{equation}
	\begin{lemma}\label{Lemma:Clement}
		There exists an operator $I_h$: $H_0^1(\Omega) \to U_h$, such that:
	\[
	\left( \sum_{K \in \mathcal{T}^h} \left\| I_h v \right\|_{k, K}^p \right)^{1/p}
	\leq C |v|_{k},
	\]
 and
	\[
	\left( \sum_{K \in \mathcal{T}^h} \left\| v - I_h v \right\|_{s, K}^p \right)^{1/p}
	\leq C h^{k-s} |v|_{s},
	\]
	for $0 \leq s \leq k \leq m$ \cite[\S 4.8]{brenner2008fembook}.
	
	There exists an $L^2$-projection $\pi_h$: $L_0^2(\Omega) \to R_h$, such that:
	
	$\quad \forall v\in L^2(\Omega)$,  $\pi_h v \in R_h$, 
	\[
	(\pi_h v ,\zeta_h)=(v, \zeta_h), \quad \forall \zeta_h \in R_h, 
	\]
	and \[
	\|v-\pi_h v\|_{0,\Omega} \leq Ch^{l+1}|v|_{l+1,\Omega},
	\]
	for $0 \leq l \leq k$. 
	\end{lemma}
	In what follows, we use $I_h$ to denote the interpolation operator for both scalar- and vector-valued functions.
	
	The finite element formulation reads: find $u_h,w_h \in U_h$, $(\boldsymbol{\phi}_h, r_h, \boldsymbol{p}_h,) \in V_h \times R_h \times Q_h$ such that
	\begin{subequations}\label{stokes8}
	\begin{equation}\label{stokes8.1}
	(\nabla w_h,\nabla v_h)=(f,v), 
	\end{equation}
	\begin{equation}\label{stokes8.2}
	B(\boldsymbol{\phi}_h, r_h, \boldsymbol{p}_h; \boldsymbol{\psi}_h, s_h, \boldsymbol{q}_h)=\mathcal{L}(\boldsymbol{\psi}_h,s_h,\boldsymbol{q}_h),
	\end{equation}
	\begin{equation}\label{stokes8.3}
	(\nabla u_h,\nabla \xi_h)=(\boldsymbol{\phi}_h,\nabla \xi_h),
\end{equation}
\end{subequations}
for all $v_h, \xi_h \in U_h$, $(\boldsymbol{\psi}_h, s_h, \boldsymbol{q}_h)\in V_h \times R_h \times Q_h $.
 
\begin{remark} 
Noting that $\operatorname{div} \boldsymbol{p}_h \in R_h$, we can eliminate $r_h$ using $a_2(\cdot, \cdot)$,  thereby reducing the actual computation to solving only for $\boldsymbol{\phi}_h$ and $\boldsymbol{p}_h$. This approach is employed in the numerical examples presented in Section \ref{sec:4}.	
	\end{remark}

	\subsection{Stability}
To prove the stability of our formulation we have to verify the two conditions, the ellipticity and the inf-sup condition. For this we will utilize the following discrete counterpart of the norm \eqref{Stokes7}:
	\begin{equation}\label{disrete1}
	\|\boldsymbol{q}\|_{\varepsilon,h}^2 = \sum_{K \in \mathcal{T}_h} \frac{h_K^2}{\varepsilon^2 + h_K^2} \norm{\operatorname{curl} \boldsymbol{q}}_{0,K}^2 + \norm{\operatorname{div} \boldsymbol{q}}_{0,K}^2.
	\end{equation}
	
	Next we will prove the inf-sup condition with this norm.
	
	\begin{lemma}\label{lemmaness} There is a constant $C>0$, independent of $h$  \text{and} $\varepsilon$, such that
	\begin{equation}
	\sup_{\boldsymbol{0} \neq (\boldsymbol{\psi},s) \in V_h \times R_h} \frac{b_1(\boldsymbol{\psi}, \boldsymbol{q}) + b_2(s, \boldsymbol{q})}{\norm{(\boldsymbol{\psi}, s)}_\varepsilon} \geq C \|\boldsymbol{q}\|_{\varepsilon,h}  \quad \forall \boldsymbol{q} \in Q_h.
	\end{equation}
	\end{lemma}
	\begin{proof} 
	  Given $\boldsymbol{q} \in Q_h$, we have $\operatorname{curl} \boldsymbol{q}|_K \in P_0(K;\mathbb{R}^3)$ and $\operatorname{div} \boldsymbol{q}|_K \in P_0(K)$. Hence, we can define $\boldsymbol{\psi} \in V_h$ and $s \in R_h$ by
	\begin{equation}
	\boldsymbol{\psi}|_K = \frac{h_K^2}{\varepsilon^2 + h_K^2} b_K \operatorname{curl} \boldsymbol{q}|_K, \quad s = \operatorname{div} \boldsymbol{q},
	\end{equation}
	where $b_K$ is the quartic bubble on $K$. For this choice of $\boldsymbol{\psi}$, we get
	\begin{equation}
	b_1(\boldsymbol{\psi}, \boldsymbol{q}) = (\boldsymbol{\psi}, \operatorname{curl} \boldsymbol{q}) \geq C \sum_{K \in \mathcal{T}_h} \frac{h_K^2}{\varepsilon^2 + h_K^2} \|\operatorname{curl} \boldsymbol{q}\|_{0,K}^2,
	\end{equation}
	and
		\begin{equation}\label{discrete2}
	\begin{aligned} 
		\|\boldsymbol{\psi}\|_\varepsilon^2 &= \varepsilon^2 \|\nabla \boldsymbol{\psi}\|_0^2 + \|\boldsymbol{\psi}\|_0^2 \leq C \sum_{K \in \mathcal{T}_h} \left( \varepsilon^2 h_K^{-2} + 1 \right) \|\boldsymbol{\psi}\|_{0,K}^2 \\
		&\leq C \sum_{K \in \mathcal{T}_h} \left( \varepsilon^2 h_K^{-2} + 1 \right) \left( \frac{h_K^2}{\varepsilon^2 + h_K^2} \right)^2 \|\operatorname{curl} \boldsymbol{q}\|_{0,K}^2 \\
		&= C \sum_{K \in \mathcal{T}_h} \frac{h_K^2}{\varepsilon^2 + h_K^2} \|\operatorname{curl} \boldsymbol{q}\|_{0,K}^2.
	\end{aligned}
\end{equation}
	By the definition of $s$, we obtain
	\begin{equation}\label{discrete3}
	b_2(s, \boldsymbol{q}) = \|\operatorname{div} \boldsymbol{q}\|_0^2 \quad \text{and} \quad \|s\|_0^2 = \|\operatorname{div} \boldsymbol{q}\|_0^2.
	\end{equation}
	Combining \eqref{discrete2}-\eqref{discrete3}, we have
		\begin{equation}
	b_1(\boldsymbol{\psi}, \boldsymbol{q}) + b_2(s, \boldsymbol{q}) \geq C \|\boldsymbol{q}\|_{\varepsilon,h}^2,
		\end{equation}
	and
	\begin{equation}
	\|(\boldsymbol{\psi}, s)\|_\varepsilon \leq C \|\boldsymbol{q}\|_{\varepsilon,h},
	\end{equation}
	which completes the proof.
	\end{proof}
	\begin{lemma}\label{lemmacrit} There is a constant $C>0$ such that
	\begin{equation}
	\sup_{\boldsymbol{0} \neq (\boldsymbol{\psi},s) \in V_h \times R_h} \frac{b_1(\boldsymbol{\psi}, \boldsymbol{q}) + b_2(s, \boldsymbol{q})}{\norm{(\boldsymbol{\psi}, s)}_\varepsilon} 
	\geq C \norm{\boldsymbol{q}}_\varepsilon^2, \quad \forall \boldsymbol{q} \in Q_h.
	\end{equation}
	\end{lemma}
	\begin{proof}
	 By Lemma \ref{lemmain}, there exist $(\boldsymbol{\psi},s) \in V \times R$ such that
	\begin{equation}
	b_1(\boldsymbol{\psi}, \boldsymbol{q}) + b_2(s, \boldsymbol{q}) \geq \norm{\boldsymbol{q}}_\varepsilon^2 \quad \text{and} \quad	\norm{(\boldsymbol{\psi}, s)}_\varepsilon \leq \norm{\boldsymbol{q}}_\varepsilon.
	\end{equation}
	With $\hat{\boldsymbol{\psi}} \in \overline{V_h}$ we denote the Clément interpolant of $\boldsymbol{\psi}$. For this there hold
		\begin{equation}
	\sum_{K \in \mathcal{T}_h} h_K^2 \norm{\boldsymbol{\psi} - \hat{\boldsymbol{\psi}}}_{0,K}^2 \leq C \norm{\nabla \boldsymbol{\psi}}_0^2,
		\end{equation}
	\begin{equation}
	\norm{\hat{\boldsymbol{\psi}}}_0 \leq C \norm{\boldsymbol{\psi}}_0 \quad \text{and} \quad \norm{\nabla \hat{\boldsymbol{\psi}}}_0 \leq C \norm{\nabla \boldsymbol{\psi}}_0.
	\end{equation}
	This gives
	\begin{equation}
	\begin{aligned}[b]
		\sum_{K \in \mathcal{T}_h} \left( \frac{\varepsilon + h_K}{h_K} \right)^2 \norm{\hat{\boldsymbol{\psi}} - \boldsymbol{\psi}}_{0,K}^2
		&\leq 2 \sum_{K \in \mathcal{T}_h} \left( \left( \frac{\varepsilon}{h_K} \right)^2 + 1 \right) \norm{\hat{\boldsymbol{\psi}} - \boldsymbol{\psi}}_{0,K}^2 \\
		&\leq C \norm{\boldsymbol{\psi}}_\varepsilon^2.
	\end{aligned}
		\end{equation}
	Using the estimates above, we obtain
		\begin{equation}
	\begin{aligned}[b]
		b_1(\hat{\boldsymbol{\psi}}, \boldsymbol{q}) &+ b_2(\pi_h s, \boldsymbol{q}) = (\hat{\boldsymbol{\psi}}, \operatorname{curl} \boldsymbol{q}) + (\pi_h s, \operatorname{div} \boldsymbol{q}) \\
		&= (\boldsymbol{\psi}, \operatorname{curl} \boldsymbol{q}) + (\hat{\boldsymbol{\psi}} - \boldsymbol{\psi}, \operatorname{curl} \boldsymbol{q}) + (s, \operatorname{div}\boldsymbol{q}) \\
		&\geq \|\boldsymbol{q}\|_\varepsilon^2 - \sum_{K \in \mathcal{T}_h} \frac{h_K}{\varepsilon + h_K} \|\operatorname{curl} \boldsymbol{q}\|_{0,K} \frac{\varepsilon + h_K}{h_K} \|\hat{\boldsymbol{\psi}} - \boldsymbol{\psi}\|_{0,K} \\
		&\geq \|\boldsymbol{q}\|_\varepsilon^2 - \|\boldsymbol{q}\|_{\varepsilon,h} \left( \sum_{K \in \mathcal{T}_h} \left( \frac{\varepsilon + h_K}{h_K} \right)^2 \|\hat{\boldsymbol{\psi}} - \boldsymbol{\psi}\|_{0,K}^2 \right)^{\frac{1}{2}} \\
		&\geq \|\boldsymbol{q}\|_\varepsilon^2 - C \|\boldsymbol{q}\|_{\varepsilon,h} \|\boldsymbol{\psi}\|_\varepsilon \\
		&\geq \left( \|\boldsymbol{q}\|_\varepsilon - C \|\boldsymbol{q}\|_{\varepsilon,h} \right) \|\boldsymbol{\psi}\|_\varepsilon,
	\end{aligned}
		\end{equation}
	where $\pi_h s\in R_h$ is the  $L^2$-projection of $s$.  Now, if $\|\boldsymbol{q}\|_\varepsilon - C \|\boldsymbol{q}\|_{\varepsilon,h} < 0$, then the assertion follows from Lemma \ref{lemmaness}. Otherwise,  they give
			\begin{equation}
	\sup_{\substack{\boldsymbol{0} \neq (\boldsymbol{\psi},s) \in V_h \times R_h}} \frac{b_1(\boldsymbol{\psi}, \boldsymbol{q}) + b_2(s, \boldsymbol{q})}{\|(\boldsymbol{\psi}, s)\|_\varepsilon} \geq C_1 \|\boldsymbol{q}\|_\varepsilon - C \|\boldsymbol{q}\|_{\varepsilon,h}.
			\end{equation}
	Combining this estimate and Lemma \ref{lemmaness}, with $0 < \alpha < 1$, we derive
	\begin{align*}
		\sup_{\substack{\boldsymbol{0} \neq (\boldsymbol{\psi},s) \in V_h \times R_h}} \frac{b_1(\boldsymbol{\psi}, \boldsymbol{q}) + b_2(s, \boldsymbol{q})}{\|(\boldsymbol{\psi}, s)\|_\varepsilon}
		&= \alpha \sup_{\substack{\boldsymbol{0} \neq (\boldsymbol{\psi},s) \in V_h \times R_h}} \frac{b_1(\boldsymbol{\psi}, \boldsymbol{q}) + b_2(s, \boldsymbol{q})}{\|(\boldsymbol{\psi}, s)\|_\varepsilon} \\
		&\quad + (1-\alpha) \sup_{\substack{\boldsymbol{0} \neq (\boldsymbol{\psi},s) \in V_h \times R_h}} \frac{b_1(\boldsymbol{\psi}, \boldsymbol{q}) + b_2(s, \boldsymbol{q})}{\|(\boldsymbol{\psi}, s)\|_\varepsilon}\\
		&\geq \alpha C_1 \norm{\boldsymbol{q}}_\varepsilon^2 - \alpha C_2 \norm{\boldsymbol{q}}_{\varepsilon,h}^2 + (1-\alpha) C \norm{\boldsymbol{q}}_{\varepsilon,h}^2 \\
		&= \alpha C_1 \norm{\boldsymbol{q}}_\varepsilon^2 + \left(C - \alpha(C + C_2)\right) \norm{\boldsymbol{q}}_{\varepsilon,h}.
	\end{align*}
	Choosing $\alpha$ such that $0 < \alpha < C/(C + C_2)$ proves the assertion.
	\end{proof}
	Lemma \ref{lemmazero} and Lemma \ref{lemmaness} give the stability result.
	
	\begin{theorem} \label{stability}
		There is a constant $C>0$ such that
			\begin{equation}
	\sup_{\substack{\boldsymbol{0} \neq (\boldsymbol{\psi}, s,\boldsymbol{q}) \in V_h \times R_h \times Q_h}}
	\frac{B(\tilde{\boldsymbol{\phi}},\tilde{r},\tilde{\boldsymbol{p}}; \boldsymbol{\psi}, s, \boldsymbol{q})}{\norm{(\boldsymbol{\psi},s)}_\varepsilon + \norm{\boldsymbol{q}}_\varepsilon}
	\ge C\left(\norm{(\tilde{\boldsymbol{\phi}},\tilde{r})}_\varepsilon + \norm{\tilde{\boldsymbol{p}}}_\varepsilon\right) 
	\quad \forall (\tilde{\boldsymbol{\phi}},\tilde{s},\tilde{\boldsymbol{p}}) \in V_h \times R_h \times Q_h.
			\end{equation}
	\end{theorem}
	Theorem	\ref{stability} implies that problem \eqref{stokes8} is well- posed. 
	
	\subsection{A priori estimate}
	The stability estimate of Theorem \ref{stability} gives the following quasi-optimality result.
	\begin{theorem}\label{priori}
		 There exists a constant $C>0$ such that
	\begin{equation}
		\begin{split}
			\|\boldsymbol{\phi} - \boldsymbol{\phi}_h\|_\varepsilon + \|r - r_h\|_0 + \|\boldsymbol{p} - \boldsymbol{p}_h\|_\varepsilon
			\leq C \Bigg\{ &\left( \inf_{\hat{\boldsymbol{\psi}} \in V_h} \|\boldsymbol{\phi} - \hat{\boldsymbol{\psi}}\|_\varepsilon
			+ \inf_{\hat{r} \in R_h} \|r - \hat{r}\|_0
			+ \inf_{\hat{\boldsymbol{q}} \in Q_h} \|\boldsymbol{p} - \hat{\boldsymbol{q}}\|_\varepsilon \right) \\
			&+ \|w - w_h\|_1 \Bigg\}.
		\end{split}
	\end{equation}
	\end{theorem}
\begin{proof}
Since $ V_h\times R_h\times Q_h \subset  V\times R\times Q$, we can get
 \[B(\boldsymbol{\phi}-\boldsymbol{\phi}_h, r-r_h, \boldsymbol{p}-\boldsymbol{p}_h; \boldsymbol{\psi}_h, s_h, \boldsymbol{q}_h) = (\nabla(w-w_h), \boldsymbol{\psi}_h).\] 
		For $\forall \hat{\boldsymbol{\phi}} \in V_h$, $\hat{r} \in R_h$, $\hat{\boldsymbol{p}} \in Q_h$, we can obtain 
		\begin{equation}\label{disrete4}
	\begin{aligned}[b]
	B(\hat{\boldsymbol{\phi}}-\boldsymbol{\phi}_h, \hat{r}-r_h, \hat{\boldsymbol{p}}-\boldsymbol{p}_h; \boldsymbol{\psi}_h, s_h, \boldsymbol{q}_h) &=B(\hat{\boldsymbol{\phi}}-\boldsymbol{\phi}, \hat{r}-r, \hat{\boldsymbol{p}}-\boldsymbol{p}; \boldsymbol{\psi}_h, s_h, \boldsymbol{q}_h) \\
	&+ (\nabla(w-w_h), \boldsymbol{\psi}_h). 
		\end{aligned}
		\end{equation}
	By Theorem \ref{stability}, for $\tilde{\boldsymbol{\phi}} = \hat{\boldsymbol{\phi}}-\boldsymbol{\phi}_h$, $ \tilde{r}=\hat{r}-r_h$, $ \tilde{\boldsymbol{p}}=\hat{\boldsymbol{p}}-\boldsymbol{p}_h$, there exists $(\boldsymbol{\psi},s,\boldsymbol{q}) \in V_h\times R_h\times Q_h $, with\[
	\|(\boldsymbol{\psi},s)\|_\varepsilon+\|\boldsymbol{q}\|_\varepsilon \leq C(\|(\boldsymbol{\phi},r)\|_\varepsilon+\|\boldsymbol{p}\|_\varepsilon ),
	\] 
	such that 
\[
\|(\tilde{\boldsymbol{\phi}}, \tilde{r})\|_\varepsilon + \|\tilde{\boldsymbol{p}}\|_\varepsilon
\leq \frac{B(\tilde{\boldsymbol{\phi}}, \tilde{r}, \tilde{\boldsymbol{p}}; \boldsymbol{\psi}, s, \boldsymbol{q})}{\|(\boldsymbol{\psi}, s)\|_\varepsilon + \|\boldsymbol{q}\|_\varepsilon}.
\]
	By \eqref{disrete4} and \eqref{boundnessB} we get \begin{align*}
	B(\tilde{\boldsymbol{\phi}}, \tilde{r}, \tilde{\boldsymbol{p}}; \boldsymbol{\psi}, s, \boldsymbol{q})
	\leq &C\left( \|\hat{\boldsymbol{\phi}} - \boldsymbol{\phi}\|_\varepsilon + \|\hat{r} - r\|_0 + \|\hat{\boldsymbol{p}} - \boldsymbol{p}\|_\varepsilon + \|\nabla(w - w_h)\|_0 \right)\\
	&\cdot \left(\|(\boldsymbol{\psi},s)\|_\varepsilon+\|\boldsymbol{q}\|_\varepsilon\right).
	\end{align*}
Therefore 
\begin{equation}
  \begin{aligned}[b]
	\|\hat{\boldsymbol{\phi}}-\boldsymbol{\phi}_h\|_\varepsilon + \|\hat{r}-r_h\|_0+ \|\hat{\boldsymbol{p}}-\boldsymbol{p}_h\|_\varepsilon &\leq C_2 \left( \|\hat{\boldsymbol{\phi}}-\boldsymbol{\phi}\|_{\varepsilon} + \|\hat{r}-r\|_{0} + \|\hat{\boldsymbol{p}}-\boldsymbol{p}\|_{\varepsilon} \right) \\
&+ \|\nabla( w - w_h)\|_0 . \end{aligned}
\end{equation}
Through the triangle inequality, we obtain 
\begin{equation}
\begin{aligned}[b]
		&\|\boldsymbol{\phi}-\boldsymbol{\phi}_h\|_\varepsilon + \|r-r_h\|_0 + \|\boldsymbol{p}-\boldsymbol{p}_h\|_\varepsilon\\
		 &\leq \|\hat{\boldsymbol{\phi}}-\boldsymbol{\phi}\|_\varepsilon + \|\hat{r}-r\|_0 + \|\hat{\boldsymbol{p}}-\boldsymbol{p}\|_\varepsilon + \|\hat{\boldsymbol{\phi}}-\boldsymbol{\phi}_h\|_\varepsilon + \|\hat{r}-r_h\|_0 + \|\hat{\boldsymbol{p}}-\boldsymbol{p}_h\|_\varepsilon \\
		&\leq (1+C_2) \left( \|\hat{\boldsymbol{\phi}}-\boldsymbol{\phi}\|_{\varepsilon} + \|\hat{r}-r\|_{0} + \|\hat{\boldsymbol{p}}-\boldsymbol{p}\|_{\varepsilon} \right) + \| w - w_h\|_1 \\
		&\leq C \left( ( \inf_{\hat{\boldsymbol{\psi}} \in V_h} \|\boldsymbol{\phi} - \hat{\boldsymbol{\psi}}\|_\varepsilon + \inf_{\hat{r} \in R_h} \|r - \hat{r}\|_0+\inf_{\hat{\boldsymbol{q}} \in Q_h} \|\boldsymbol{p} - \hat{\boldsymbol{q}}\|_\varepsilon )+\|w-w_h\|_1\right). 
	\end{aligned}
	\end{equation}
	\end{proof}	
   Combining the above theorem with standard interpolation estimates and the fact that $w_h$ is the finite element approximation to $w$, we can now prove the \emph{a priori} error estimate.
	\begin{theorem} \label{norm}
		Let $(w, \boldsymbol{\phi},r,\boldsymbol{p}) \in U\times V \times R \times Q$ be the solution of problem \eqref{stokes6.1}-\eqref{stokes6.2},  and $(w_h, \boldsymbol{\phi}_h,r_h,\boldsymbol{p}_h) \in U_h\times V_h \times R_h \times Q_h$ be the solution of \eqref{stokes8.1}-\eqref{stokes8.2}. Assume $(w,\boldsymbol{\phi}, r, \boldsymbol{p}) \in H^2(\Omega) \times H^2(\Omega;\mathbb R^3) \times H^1(\Omega)\times H^2(\Omega;\mathbb R^3)$. We have
	\begin{equation}
	\|\boldsymbol{\phi} - \boldsymbol{\phi}_h\|_\varepsilon + \|r - r_h\|_0 + \|\boldsymbol{p} - \boldsymbol{p}_h\|_\varepsilon \leq C(\varepsilon h+h^2)\|\boldsymbol{\phi}\|_2+Ch\|r\|_1+Ch\|\boldsymbol{p}\|_2+Ch\|w\|_2.
	\end{equation}
	\end{theorem}
	\begin{proof} 
		By the standard interpolation error, we know
	\begin{equation}
	\inf_{\hat{\boldsymbol{\psi}} \in V_h} \|\boldsymbol{\phi} - \hat{\boldsymbol{\psi}}\|_\varepsilon \leq \|\boldsymbol{\phi}-I_h\boldsymbol{\phi}\|_\varepsilon,
\end{equation}
	and 
	\begin{equation}
	\begin{aligned}[b]
	\|\boldsymbol{\phi}-I_h\boldsymbol{\phi}\|_{\varepsilon}^2&=\varepsilon^2\|\nabla (\boldsymbol{\phi}-I_h\boldsymbol{\phi})\|_0^2+\|\boldsymbol{\phi}-I_h\boldsymbol{\phi}\|_0^2\\
	&\leq \varepsilon^2(C_1h|\boldsymbol{\phi}|_2))^2+(C_2h^2\|\boldsymbol{\phi}\|_2)^2\\
	&\leq C_0(\varepsilon^2h^2+h^4)\|\boldsymbol{\phi}\|_2^2. 
	\end{aligned}
	\end{equation}
Therefore \begin{equation}\label{phi}
\inf_{\hat{\boldsymbol{\psi}} \in V_h} \|\boldsymbol{\phi} - \hat{\boldsymbol{\psi}}\|_\varepsilon \leq C(\varepsilon h+h^2)\|\boldsymbol{\phi}\|_2. 
\end{equation}
Similarly	\begin{equation}\label{discrete5}
\inf_{\hat{r} \in R_h} \|r - \hat{r}\|_0 \leq \|r-I_hr\|_0 \leq Ch\|r\|_1=0, \end{equation}
\begin{equation}\label{discrete6}
\inf_{\hat{\boldsymbol{q}} \in Q_h} \|\boldsymbol{p} - \hat{\boldsymbol{q}}\|_\varepsilon \leq \|\boldsymbol{p} - I_h\boldsymbol{p}\|_\varepsilon ,
\end{equation}
and 
\begin{equation}
\begin{aligned}[b]
	\|\boldsymbol{p} - I_h\boldsymbol{p}\|_\varepsilon^2 &= \left( \sup_{\boldsymbol{\psi} \in V} \frac{\langle \boldsymbol{\psi}, \operatorname{curl} (\boldsymbol{p} - I_h\boldsymbol{p}) \rangle}{\|\boldsymbol{\psi}\|_\varepsilon} \right)^2 + \|\operatorname{div} (\boldsymbol{p} - I_h\boldsymbol{p})\|_0^2\\
	&\leq \|\operatorname{curl} (\boldsymbol{p} - I_h\boldsymbol{p})\|_0^2+\|\operatorname{div} (\boldsymbol{p} - I_h\boldsymbol{p})\|_0^2\\
	&\leq C_3h^2\|\boldsymbol{p}\|_2^2 .
	\end{aligned}
	\end{equation}
	Combining the above estimates, we obtain\begin{equation}\label{discrete7}
	\inf_{\hat{\boldsymbol{q}} \in Q_h} \|\boldsymbol{p} - \hat{\boldsymbol{q}}\|_\varepsilon \leq Ch\|\boldsymbol{p}\|_2 .
	\end{equation}
Consequently, the desired result follows from \eqref{phi}, \eqref{discrete5}, and \eqref{discrete7}, together with the fact that $w_h$ is the finite element approximation to $w$.
	\end{proof}
	
	\begin{theorem}\label{Th:errorU}
			Let $(w,\boldsymbol{\phi},r,\boldsymbol{p},u) $ be the solution of problem \eqref{stokes6.1}-\eqref{stokes6.3},  $u_h$ is the solution of \eqref{stokes8.3}. Under the assumption of Theorem \ref{norm}, We have\[
	|u - u_h|_1\lesssim h \left( |u|_2 + \|\boldsymbol{\phi}\|_2 + \|r\|_1 + \|\boldsymbol{p}\|_2 + \|w\|_2 \right).
	\]
	\end{theorem}
	\begin{proof} 
		It follows from \eqref{stokes6.3} that
	\[
		(\nabla (u - u_h), \nabla \xi_h) = (\boldsymbol{\phi} - \boldsymbol{\phi}_h, \nabla \xi_h) \quad \forall \xi_h \in V_h.\]
		Taking $\xi_h=\hat{u}-u_h$, where $\hat{u}\in V_h$ is arbitrary, gives
        \[(\nabla (\hat{u} - u_h), \nabla (\hat{u} - u_h)) = (\boldsymbol{\phi} - \boldsymbol{\phi}_h, \nabla (\hat{u} - u_h)) + (\nabla (\hat{u} - u), \nabla (\hat{u} - u_h)).\]
        Hence
		\begin{equation}\label{include}
		 |\hat{u} - u_h|_1 \leq C_3 \left( \|\boldsymbol{\phi} - \boldsymbol{\phi}_h\|_0 + |\hat{u} - u|_1 \right).		 
		  \end{equation}
By the triangle inequality, we obtain
\begin{equation}
\begin{aligned}[b]
		|u - u_h|_1 &\leq |u - \hat{u}|_1 + |\hat{u} - u_h|_1 \\
		& \leq (1 + C_3) |u - \hat{u}|_1 + C_3 \|\boldsymbol{\phi} - \boldsymbol{\phi}_h\|_0 \\
		& \lesssim h |u|_2 + \|\boldsymbol{\phi} - \boldsymbol{\phi}_h\|_0 \\
		& \lesssim h \left( |u|_2 + \|\boldsymbol{\phi}\|_2 + \|r\|_1 + \|\boldsymbol{p}\|_2 + \|w\|_2 \right).
	\end{aligned}
	\end{equation}
	\end{proof}
	
	\begin{remark}
		 In this paper, $u$ and $w$ are approximated using linear elements, so we directly use the fact $\|\boldsymbol{\phi}-\boldsymbol{\phi}_h\|_0 \leq \|\boldsymbol{\phi}-\boldsymbol{\phi}_h\|_\varepsilon$. However, if quadratic elements are used to approximate $u$ and $w$, then $\|\boldsymbol{\phi}-\boldsymbol{\phi}_h\|_0$ needs to be handled via duality estimates, which is omitted here due to space limitations. 
	\end{remark}
We next investigate the error estimate in the limiting case as $\varepsilon \to 0$. To this end, by setting $\varepsilon= 0$ in \eqref{equation1}, the original problem reduces to the following Poisson problem:
\begin{equation}\label{poisson}
	\begin{cases} 
	-\Delta u^0 = f & \text{in } \Omega, \\ 
				u^0 = 0 & \text{on } \partial\Omega.
	\end{cases}
\end{equation}
In analogy with \cite{cui2025loworderfiniteelementcomplex}, we make the following assumption:
Assume that \eqref{equation1} and \eqref{poisson} satisfy the regularity \begin{equation}\label{hypothesis1}
	\begin{aligned}
	&\|u^0\|_2 = \|w\|_2 \lesssim \|f\|_0,\\
	&\|u - u^0\|_1 + \varepsilon \|u\|_2 + \varepsilon^2 \|u\|_3 \lesssim \varepsilon^{1/2} \|f\|_0.
	\end{aligned}
\end{equation}
Under above estimate, we have \cite{cui2025loworderfiniteelementcomplex}
\begin{subequations}\label{hypothesis2}
\begin{equation}\label{hypothesis2.1}
	\|\boldsymbol{\phi}^0\|_1 \lesssim \|f\|_0,
	\end{equation}
	\begin{equation}\label{hypothesis2.2}
		\|\boldsymbol{\phi} - \boldsymbol{\phi}^0\|_0 + \varepsilon \|\boldsymbol{\phi}\|_1 + \varepsilon^2 \|\boldsymbol{\phi}\|_2 \lesssim \varepsilon^{1/2} \|f\|_0,
		\end{equation}
		\begin{equation}\label{hypothesis2.3}
			\|\operatorname{curl } \boldsymbol{p}\|_0 \lesssim \varepsilon^{1/2} \|f\|_0,
			\end{equation}
\end{subequations}
where $\boldsymbol{\phi}^0=\nabla u$. 

With the above assumption and regularity estimates, we are now in a position to state our main results.
	\begin{theorem} \label{conclusion}
		 	Let $(\boldsymbol{\phi},r,\boldsymbol{p}) \in V \times R \times Q$ be the solution of problem \eqref{stokes6.2},  and $( \boldsymbol{\phi}_h,r_h,\boldsymbol{p}_h) \in U_h \times V_h \times R_h \times Q_h$ be the solution of \eqref{stokes8.2}. Under the above assumptions, we have
		 	\begin{equation}\label{layer}
		 		\|\boldsymbol{\phi} - \boldsymbol{\phi}_h\|_{\varepsilon} + \|r - r_h\|_0 + \|\boldsymbol{p} - \boldsymbol{p}_h\|_{\varepsilon} \leq Ch^{1/2}\|f\|_0,
		 	\end{equation}
         where C is a constant independent of \( \varepsilon \), \( h \) and \( f \).   

	\end{theorem}
	\begin{proof}
	  We first show that
	\begin{equation}
	\qquad \inf_{v \in W_h} \| \boldsymbol{\phi} - v \|_{\varepsilon} \leq \| \boldsymbol{\phi} - I_h \boldsymbol{\phi} \|_{\varepsilon} \leq C h^{1/2} \| f \|_0.
	\end{equation}
	Let $I_h\boldsymbol{\phi}\in V_h$ denote the Clément-Scott-Zhang interpolant of $\boldsymbol{\phi}$ introduced in Lemma \ref{Lemma:Clement}. We then obtain
	\begin{equation}
	\begin{aligned}[b]
		\varepsilon \| \boldsymbol{\phi} - I_h \boldsymbol{\phi} \|_1 &\leq C \varepsilon \| \boldsymbol{\phi} \|_1^{1/2} \| \boldsymbol{\phi} - I_h \boldsymbol{\phi} \|_1^{1/2} \\
		&\leq C \varepsilon h^{1/2} \| \boldsymbol{\phi} \|_1^{1/2} |\boldsymbol{\phi}|_2^{1/2} \\
		&\leq C h^{1/2} \| f\|_0. 
	\end{aligned}
	\end{equation}
	In order to estimate the $L^2$-part of the energy norm, we use the triangle inequality to obtain
	\begin{equation}
	\| \boldsymbol{\phi} - I_h \boldsymbol{\phi} \|_0 \leq \| \boldsymbol{\phi} - \boldsymbol{\phi}^0 - I_h(\boldsymbol{\phi} - \boldsymbol{\phi}^0) \|_0 + \| \boldsymbol{\phi}^0 - I_h \boldsymbol{\phi}^0 \|_0.
	\end{equation}
	By Lemma \ref{Lemma:Clement}, \eqref{hypothesis2.1} and \eqref{hypothesis2.2} we have
	\begin{equation}
	\begin{aligned}[b]
    \|\boldsymbol{\phi}-\boldsymbol{\phi}^0-I_h(\boldsymbol{\phi}-\boldsymbol{\phi}^0)\|_0 &=\|\boldsymbol{\phi}-\boldsymbol{\phi}^0-I_h(\boldsymbol{\phi}-\boldsymbol{\phi}^0)\|_0^{1/2}\|\boldsymbol{\phi}-\boldsymbol{\phi}^0-I_h(\boldsymbol{\phi}-\boldsymbol{\phi}^0)\|_0^{1/2}\\
    &\leq \|\boldsymbol{\phi}-\boldsymbol{\phi}^0\|_0^{1/2}Ch^{1/2} \|\boldsymbol{\phi}-\boldsymbol{\phi}^0\|_1^{1/2}  \\
    &\leq C\varepsilon^{1/2}h^{1/2}\|f\|_0^{1/2}\|\boldsymbol{\phi}-\boldsymbol{\phi}^0\|_1^{1/2}\\
    &\leq Ch^{1/2}\|f\|_0, 
\end{aligned}
    \end{equation}
and 
    	\begin{equation}
    \|\boldsymbol{\phi}^0-I_h\boldsymbol{\phi}^0\|_0 \leq C h\|\boldsymbol{\phi}^0\|_1 \leq Ch\|f\|_0.
		\end{equation}
	From the above estimates and Theorem \ref{priori} we get
	\begin{equation}
	\|\boldsymbol{\phi}-I_h\boldsymbol{\phi}\|_{\varepsilon} \leq C h^{1/2} \|f\|_0 +C h^{1/2} \|f\|_0+Ch\|f\|_0\leq Ch^{1/2} \|f\|_0.
	\end{equation}
	Clearly, $\boldsymbol{0} \in Q_h$, so we have 
	\begin{equation}
	\inf_{\hat{\boldsymbol{q}} \in Q_h} \| \boldsymbol{p} - \hat{\boldsymbol{q}} \|_\varepsilon \leq \| \boldsymbol{p} \|_\varepsilon. 
	\end{equation}
	For $\varepsilon >0 $, by \eqref{Stokes7}, \eqref{Stokes9}, \eqref{hypothesis2.3}, \cite[Lemma\ 3.6]{Vivette1986finite} and $\mathrm{div}\, \boldsymbol{p}=0$, we get
	\begin{equation}
	\| \boldsymbol{p} \|_\varepsilon \leq \| \boldsymbol{p} \|_0 \leq C \| \mathrm{curl}\, \boldsymbol{p} \|_0 \leq C \varepsilon^{\frac{1}{2}} \| f \|_0. 
	\end{equation}
	Therefore, if $\varepsilon \ll h$, we have
	\begin{equation}
	\inf_{\hat{\boldsymbol{q}} \in Q_h} \| \boldsymbol{p} - \hat{\boldsymbol{q}} \|_\varepsilon \leq C \varepsilon^{\frac{1}{2}} \| f \|_0 \leq C h^{\frac{1}{2}} \| f \|_0. 	
	\end{equation}
	In conclusion, \begin{equation}
		\|\boldsymbol{\phi} - \boldsymbol{\phi}_h\|_{\varepsilon} + \|r - r_h\|_0 + \|\boldsymbol{p} - \boldsymbol{p}_h\|_{\varepsilon} \leq Ch^{1/2} \| f \|_0. 
	\end{equation}
	\end{proof}
    
The following theorem is an immediate consequence of \eqref{include}, \eqref{layer}, and the triangle inequality.
\begin{theorem}\label{Th:errULayer}
Let $(w,\boldsymbol{\phi},r,\boldsymbol{p},u) $ be the solution of problem \eqref{stokes6.1}-\eqref{stokes6.3} and $u_h$ the solution of \eqref{stokes8.3}. Under the assumption of Theorem \ref{conclusion}, we have
\begin{equation}
|u-u_h|_1 \leq Ch^{1/2}\|f\|_0. 
\end{equation}
\end{theorem}

\begin{remark}
It should be emphasized that our method is computationally inexpensive, especially compared to constructing complex three-dimensional $C^1$-conforming elements. 
The decomposed problems consist of two Poisson equations and one Brinkman-type equation, which can be solved using standard finite elements and fast solvers. 
Additionally, the newly introduced variable $\boldsymbol\phi$  corresponds to the gradient of the original variable $u$ and the $\varepsilon$-norm of $\boldsymbol\phi$  is equivalent to a 
weighted $H^2$-norm of $u$ and attains the same convergence order as $u$. 
\end{remark}

\section{Numerical Results}\label{sec:4}
	This section presents numerical results to validate the proposed decoupled mixed finite element method. Let $\Omega$ be the unit cube domain $(0,1)^3$. The mesh $\mathcal{T}_h$ is constructed by first dividing the unit cube into $N^3$ small cubes, with $N=4,8,16,32,64$, and then subdividing each small cube into six tetrahedra.
All the following numerical results were obtained using FreeFEM++ \cite{MR3043640}.
	
	\subsection{Numerical test without boundary layer}  We consider the fourth-order
	elliptic singular perturbation problem \eqref{equation1} with the exact solution given by
	\[
	u=\sin(\pi x)^2\sin(\pi y)^2\sin(\pi z)^2.  
	\]
	The right-hand term can be computed from equation \eqref{equation1}.
	
	Since $w$ and $\boldsymbol{p}$ cannot be obtained from the expression of $u$, we compute only the errors of $u$ and $\boldsymbol{\phi}$ here for simplicity. The corresponding numerical errors are summarized in Table 1, including $|u-u_h|_1$ and  $\|\boldsymbol{\phi}-\boldsymbol{\phi}_h\|_\varepsilon$ for various values of $\varepsilon$ and mesh size $h$. 

	As shown in Table \ref{tab:errors_convergence}, $|u-u_h|_1$ and $\|\boldsymbol{\phi}-\boldsymbol{\phi}_h\|_\varepsilon$ exhibit the first-order convergence for $\varepsilon=1$. As $\varepsilon$ decreases to 0.001 and $10^{-6}$,  the convergence rate $|u-u_h|_1$ and $\|\boldsymbol{\phi}-\boldsymbol{\phi}_h\|_\varepsilon$ remain $\mathcal{O}(h^1)$. The above results are consistent with Theorem \ref{norm} and Theorem \ref{Th:errorU}.
	\begin{table}[!htbp]\centering
		\caption{Errors$|u-u_h|_1$ and $\|\boldsymbol{\phi}-\boldsymbol{\phi}_h\|_\varepsilon$ }
		\label{tab:errors_convergence}
		\begin{tabular}{c c c c c c c c c c}
			\toprule
			$\varepsilon$ & $h$ & $|u-u_h|_1$ & $rate$ & $\|\boldsymbol{\phi}-\boldsymbol{\phi}_h\|_{\varepsilon}$ & $rate$ \\
			\midrule
			\multirow{5}{*}{$1$}
			& 4.330e-01 & 1.065e+00 & --- & 8.105e+00 & ---  \\
			& 2.165e-01 & 5.542e-01 & 0.94 & 4.591e+00 & 0.82  \\
			& 1.083e-01 & 2.622e-01 & 1.08 & 2.395e+00 & 0.94  \\
			& 5.413e-02 & 1.273e-01 & 1.04 & 1.234e+00 & 0.96  \\
			& 2.706e-02 & 6.300e-02 & 1.01 & 6.130e-01 & 1.00  \\
			\midrule
			\multirow{5}{*}{$0.001$}
			& 4.330e-01 & 8.757e-01 & --- & 7.625e-01 & ---  \\
			& 2.165e-01 & 4.857e-01 & 0.85 & 3.659e-01 & 1.06 \\
			& 1.083e-01 & 2.486e-01 & 0.97 & 1.700e-01 & 1.11  \\
			& 5.413e-02 & 1.248e-01 & 0.99 & 7.743e-02 & 1.13  \\
			& 2.706e-02 & 6.247e-02 & 1.00 & 3.145e-02 & 1.32 \\
			\midrule
			\multirow{5}{*}{$10^{-6}$}
			& 4.330e-01 & 8.754e-01 & --- & 7.628e-01 & --- \\
			& 2.165e-01 & 4.854e-01 & 0.85 & 3.677e-01 & 1.05  \\
			& 1.083e-01 & 2.484e-01 & 0.97 & 1.742e-01 & 1.08  \\
			& 5.413e-02 & 1.248e-01 & 0.99 & 8.524e-02 & 1.03  \\
			& 2.706e-02 & 6.244e-02 & 1.00 & 4.236e-02 & 1.01  \\
			\bottomrule
		\end{tabular}
	\end{table}
\subsection{Numerical test with boundary layer}  We next verify the convergence
behavior of the conforming finite element methods \eqref{stokes8} for problem \eqref{equation1} with boundary layers.

	To this end, consider the exact solution to the Poisson equation \eqref{poisson}:
	\[
	u^0=\sin(\pi x)\sin(\pi y)\sin(\pi z).
	\]
	The corresponding right-hand side $f$ computed from the Poisson equation is used in \eqref{equation1}, for which
	the exact solution $u$ is not known in closed form. When $\varepsilon$ becomes smaller, the solution $u$ develops pronounced boundary layers. From Table \ref{tab:errors_convergence2}, it can be clearly observed that the convergence orders of $\|\boldsymbol{\phi}^0-\boldsymbol{\phi}_h\|_\varepsilon$ tend to be of order $1/2$, which is consistent with Theorem \ref{conclusion} and Theorem \ref{Th:errULayer}. In contrast, the convergence orders of $|u^0-u_h|_1$ are slightly higher than $1/2$ on the tested meshes, but still remain consistent with the theoretical prediction. 
	\begin{table}[!htbp]
		\centering
		\caption{Errors $|u^0-u_h|_1$ and $\|\boldsymbol{\phi}^0-\boldsymbol{\phi}_h\|_\varepsilon$ }
		\label{tab:errors_convergence2}
		\begin{tabular}{c c c c c c c c c c}
			\toprule
		$\varepsilon$ & $h$ & $|u^0-u_h|_1$ & $rate$ & $\|\boldsymbol{\phi}^0-\boldsymbol{\phi}_h\|_{\varepsilon}$ & $rate$  \\
				\midrule
		\multirow{5}{*}{$10^{-6}$}
		& 4.330e-01 & 9.682e-01 & --- & 1.024e+00 & ---  \\
		& 2.165e-01 & 5.197e-01 & 0.90 & 6.074e-01 & 0.75  \\
		& 1.083e-01 & 2.754e-01 & 0.92 & 3.843e-01 & 0.66  \\
		& 5.413e-02 & 1.517e-01 & 0.86 & 2.558e-01 & 0.59  \\
		& 2.706e-02 & 8.832e-02 & 0.78 & 1.752e-01 & 0.55 \\
		\midrule
		\multirow{5}{*}{$10^{-8}$}
		& 4.330e-01 & 9.682e-01 & --- & 1.024e+00 & ---  \\
		& 2.165e-01 & 5.197e-01 & 0.90 & 6.074e-01 & 0.75 \\
		& 1.083e-01 & 2.754e-01 & 0.92 & 3.843e-01 & 0.66  \\
		& 5.413e-02 & 1.517e-01 & 0.86 & 2.558e-01 & 0.59  \\
		& 2.706e-02 & 8.832e-02 & 0.78 & 1.753e-01 & 0.55 \\
		\midrule
		\multirow{5}{*}{$10^{-10}$}
			& 4.330e-01 & 9.682e-01 & --- & 1.024e+00 & ---  \\
		& 2.165e-01 & 5.197e-01 & 0.90 & 6.074e-01 & 0.75 \\
		& 1.083e-01 & 2.754e-01 & 0.92 & 3.843e-01 & 0.66  \\
		& 5.413e-02 & 1.517e-01 & 0.86 & 2.558e-01 & 0.59  \\
		& 2.706e-02 & 8.832e-02 & 0.78 & 1.753e-01 & 0.55 \\
		\bottomrule
		\end{tabular}
	\end{table}

\vspace{5mm}

\section*{Acknowledgments}
The authors would like to thank anonymous referees for their valuable comments. This work is supported by the National Natural Science Foundation of China (Grant No. 12301529), Yunnan Fundamental Research Projects (No. 202401AU070214) and Yunnan Xingdian Talent Support Program - Young Talent Project.

		\bibliographystyle{abbrv}
		\bibliography{Reference}
		
\end{document}